\documentclass[twoside,11pt,reqno]{amsart}

\usepackage{eurosym}
\usepackage{amsmath,amsthm,amscd,amsfonts,amssymb,enumerate}
\usepackage{graphicx,color}
\usepackage{multirow}
\usepackage{mathrsfs, cite, url}
\usepackage[colorlinks=true,allcolors=blue]{hyperref}
\usepackage{comment}

\newcommand{\sign}{\mathop{\rm sign}}

\newtheorem{theorem}{Theorem}[section]

\newtheorem{corollary}{Corollary}[section]

\newtheorem{example}{Example}[section]
\newtheorem{remark}{Remark}[section]

\numberwithin{equation}{section}

\newcommand{\Ii}[5]{\mathbf{I}^{{#1}}(#2,#3,#4,#5)}
\newcommand{\Inormalized}[5]{\mathbf{J}}

\begin{document}

\title{A family of polylogarithmic integrals}

\author[A. Sofo, J-C. Pain, V. Scharaschkin]{Anthony Sofo, Jean-Christophe Pain$^{\ast}$, Victor Scharaschkin}
\address{Anthony Sofo, {College of Engineering and Science, Victoria
University, Australia}}
\address{Adjunct Professor, {School of Science, RMIT University, Australia}}
\email{anthony.sofo@vu.edu.au}
\address{Jean-Christophe Pain, CEA, DAM, DIF, F-91297 Arpajon, France}
\address{Universit\'e Paris-Saclay, CEA, Laboratoire Mati\`ere en Conditions Extr\^emes,
F-91680, Bruy\`eres-le-Ch\^atel, France}
\email{jean-christophe.pain@cea.fr}
\address{Victor Scharaschkin}
\address{Department of Mathematics. University of Queensland. St Lucia, Australia, 4072}
\email{v.scharaschkin@gmail.com}
\date{August, 2024}
\subjclass[2010]{Primary 11M06, 11M35, 26B15; Secondary 33B15, 42A70, 65B10.}
\keywords{Integrals, Harmonic number, Polylogarithm function, Bernoulli
number}
\thanks{$^{\ast}$ Corresponding author}

\begin{abstract}
In this paper we investigate a class of integrals that were encountered in the course of a work on statistical plasma physics, in the so-called Sommerfeld temperature-expansion of the electronic entropy. We show that such integrals, involving some parameters, can be fully described in closed form represented by special functions.
\end{abstract}


\maketitle


\section{Introduction} 

\bigskip Let $\mathbb{N}$, $\mathbb{R}$, $\mathbb{C}$ denote the sets of natural numbers, real numbers and complex numbers respectively. In this paper we focus on evaluating the integral 
\begin{equation}
\Ii{\pm}{a}{b}{p}{t} :=\int_{-\infty }^{\infty }\,\frac{x^{p}~
\mathrm{Li}_{t}(\pm e^{ax})}{1+e^{bx}}\,\mathrm{d}x,  \label{ONE}
\end{equation}
for $a$, $b\in \mathbb{R}$ with $ab>0$, and $p\in \mathbb{N}\cup\{0\}$.
 
The integral~(\ref{ONE}) was encountered for specific values of the parameters $(a, b, p, t)$ during research in statistical plasma physics, specifically in the context of the Sommerfeld temperature expansion of electronic entropy. This formalism is crucial in calculating the equation of state for dense plasmas. For further details, see~\cite{ARNAU}. In particular, an integral representation
for $\zeta (4)$ was obtained in~\cite{PAIN1}, and a result for $a=b=t=1$ and arbitrary~$p$ was derived in~\cite{KUMAR}, involving harmonic sums. See also~\cite{LIChu} for similar results.

Other cases have been published by Alzer and Choi~\cite{Al-Ch}, and by Sofo and Choi~\cite{SofChoi}. Euler sums, along with the closely related field of multiple zeta values, have recently garnered significant attention in research, as evidenced by studies such as~\cite{Bo-Br-Ka},~\cite{Erdel}, and~\cite{Lewin}. Similarly, various parameterized Euler sums have also been
explored extensively; see, for instance,~\cite{SofoNimb2019},~\cite{SoNimIntTrans}
,~\cite{SriChoi2001}, and~\cite{Sr-Ch-12}.

In this paper we shall extend the results of~\cite{KUMAR} by representing $I^{\pm}(a,b,p,t)$ quite generally in terms of special functions, such as the Riemann zeta function, $\zeta$, and linear harmonic Euler sums.  Our main results are as follows, with the notation defined in \S\ref{Sec-Notation} below.

\begin{theorem}
\label{THM1}  
Let $t \in \mathbb{N}$, $p\in \mathbb{N}\cup\{0\}$ and $a$, $b \in \mathbb{R}$ with $ab>0$.   Define $q = \frac{a}{b}$.  Then, the following identities hold.  If $p+t$ is odd
\begin{equation*}
\frac{\sign(b)}{p!}b^{p+1}\, \Ii{-}{a}{b}{p}{t}
= \frac{(-1)^{p+1}}{q^{p+1}}\eta(p+t+1) \;\; -2\sum\limits_{j=0}^{\lfloor \frac{t}{2} \rfloor } q^{t-2j} \binom{p\!+\!t\!-\!2j}{p} \eta(p+t+1-2j)\,\eta(2j) %
\end{equation*}

\noindent If $p+t$ is even
\begin{align}
\frac{\sign(b)}{p!}b^{p+1} \,\Ii{-}{a}{b}{p}{t}
=& \frac{(-1)^{p+1}}{q^{p+1}}\eta(p+t+1) \nonumber\\
&-2\sum\limits_{j=0}^{\lfloor \frac{t}{2} \rfloor }q^{t-2j} \binom{p\!+\!t\!-\!2j}{p} \eta(p+t+1-2j)\,\eta(2j) \nonumber\\ %
&+2(-1)^{p+1}\Bigl[ \mathbb{S}_{p+1,t}^{+-}(q) -2^{-p}\,\mathbb{S}_{p+1,t}^{+-} (\textstyle \frac{q}{2}) -\eta(p+1)\,\eta(t) \Bigr] %
.\label{INTthm1.4} 
\end{align}
Here $\mathrm{Li}_{t}$ is the
polylogarithm function, $\eta$ is the Dirichlet eta
function, and $\mathbb{S}_{p+1,t}^{+-}(q)$ and $\mathbb{S}_{p+1,t}^{+-}(\frac{q}{2})$
are the linear harmonic Euler sums defined by~(\ref{FlajSal3}) below.  See~\S{\ref{Sec-Notation}} for detailed definitions.
\end{theorem}

The results for the complex-valued integral $\Ii{+}{a}{b}{p}{t}$ are analogous and are presented and proven in Theorem~\ref{THM2} below.

\section{Notation and Preliminary Results} \label{Sec-Notation}

In this section, we provide a detailed definition of the terms used in Theorem~\ref{THM1} and review their properties. First, we define various species of linear harmonic Euler sums.  For $n$, $p\in \mathbb{N}$ recall that the $n$th \emph{harmonic number} of \emph{order $p$}, $H_n^{(p)}$ is defined by 
\begin{equation}
H_{n}^{(p)}:=\sum\limits_{j=1}^{n}\frac{1}{j^p} \qquad
 p, n\in \mathbb{N}.
\label{HN-ext}
\end{equation} 
We write $H_n$ for $H_n^{(1)}$, and we set $H_0^{(p)}:=0$.

Euler initiated the study of linear harmonic sums of the form
\begin{equation*}
\mathbb{S}_{p,t}=\sum_{n\geq 1}^{{}}\,\frac{H_{n}^{(p)}}{n^{t}} \qquad t>1,
\end{equation*}
with the sum $p+t$ said to be the \emph{weight} of $\mathbb{S}_{p,t}$.  
Many researchers have continued this work.  See for example~\cite{Al-Ch},~\cite{Borwein},~\cite{FlajSalv} and~\cite{SofINTtrans}.  It is now known that $\mathbb{S}_{p,t} $ can be explicitly evaluated in terms of special functions, such as
the Riemann zeta function, in the following cases: when $p=t$; 
when $p+t$ is odd; and the exceptional cases $(p,t) =(2, 4)$ and $(4, 2)$.
Furthermore, an elegant \emph{shuffle} (reciprocity) relation 
\begin{equation*}
\mathbb{S}_{p,t}+\mathbb{S}_{t,p}=\zeta(p)\zeta(t) +\zeta(p+t)
\end{equation*}
relates $\mathbb{S}_{t,p}$ and $\mathbb{S}_{p,t}$.   Flajolet and Salvy 
\cite{FlajSalv} studied harmonic sums with alternating signs such as
\begin{equation}
\mathbb{S}_{p,t}^{+-} :=\sum_{n=1}^{\infty }\,(-1)^{n+1}\frac{H_{n}^{(p)}}{n^{t}}.  \label{FlajSal2} 
\end{equation}
(Our $\mathbb{S}_{p,t}$ may be denoted $\mathbb{S}_{p,t}^{++}$ in this context.)  We generalize further, by defining for real $r>0$
\begin{equation}
\mathbb{S}_{p,t}(r) :=\sum_{n=1}^{\infty } \frac{H_{rn}^{(p)}}{n^{t}}, \qquad \mathbb{S}_{p,t}^{+-}(r) :=\sum_{n=1}^{\infty }\,(-1)^{n+1}\frac{H_{rn}^{(p)}}{n^{t}} \label{FlajSal3} 
\end{equation}
where the $H_{\lambda}^{(p)}$ denote the \emph{extended harmonic numbers} of \emph{index} 
$\lambda$ (see~\cite{SofoSri}) which are defined quite generally for $\lambda \in \mathbb{C}\setminus \{ \cdots
,-3,-2,-1,0\}$
by 
\begin{equation}
H_{\lambda}^{(p)} :=\begin{cases}
\gamma +\psi (\lambda +1) & p=1, \\[2mm] 
\zeta (p)+\frac{(-1)^{p-1}}{(p-1)!\vphantom{\int^2}}\, \psi^{(p-1)}(\lambda +1) \quad\qquad & p\in \mathbb{N},\;\; p
\geq 2.
\end{cases}
\label{G-HZ}
\end{equation}
Here $\zeta$ is the Riemann zeta function and $\psi$ denotes the
\emph{digamma} (or \emph{psi}) function defined by 
\begin{equation*}
\psi(z) :=\frac{d}{dz}\bigl(\log \Gamma(z)
\bigr) =\frac{\Gamma^{\prime }(z) }{\Gamma(z) }
\qquad z\in \mathbb{C}\setminus \{ \cdots
,-3,-2,-1,0\},  
\end{equation*}
where $\Gamma$ is the usual Gamma function (see, e.g., 
\cite[Section 1.1]{Sr-Ch-12}).   The term $\gamma$ represents the familiar Euler-Mascheroni constant $\gamma = 0.5772\ldots$ (see,
e.g.,~\cite[Section 1.2]{Sr-Ch-12}),  which can also be obtained as $\gamma=-\psi(1)$. Finally, for $k\in \mathbb{N}$ the $k$th \emph{polygamma} function $\psi^{(k)}$ is defined by 
\begin{eqnarray}
\psi^{(k)}(z) &:= &\frac{d^{k}}{dz^{k}}\bigl(\psi (z) \bigr) =(
-1)^{k+1}\, k!\, \sum_{r=0}^{\infty }\frac{1}{(r+z)^{k+1}}
\label{polygamma} \\
&& \qquad k\in \mathbb{N},\,\,z\in \mathbb{C}\setminus \{ \cdots
,-3,-2,-1,0\}.  \notag
\end{eqnarray}
(We also put $\psi^{(0)}=\psi$). 

Polygamma functions satisfy the recurrence relation
\begin{equation}\label{eq-Polygamma}
\psi^{(k)}(z+1) = \psi^{(k)}(z) + (-1)^k\,k!\, z^{-1-k} \qquad k\in \mathbb{N}\cup\{0\}.
\end{equation}
In particular when $k=0$ repeated application of~(\ref{eq-Polygamma}) reveals a close connection of the digamma function with harmonic numbers (see, e.g.,~\cite[Section 1.3]
{Sr-Ch-12}), namely 
\begin{equation}
\psi(z+n)=\psi (z)+\sum_{k=1}^{n}\frac{1}{z+k-1},\qquad n\in \mathbb{N}, \quad z\in \mathbb{C}\setminus \{ \cdots
,-3,-2,-1,0\}.
\label{Psi-a}
\end{equation}
Furthermore, combining~(\ref{eq-Polygamma}) and~(\ref{G-HZ}) we obtain
\begin{equation}\label{eq-difference-harmonic}
H_{\lambda+1}^{(p)} -H_{\lambda}^{(p)} =\frac{1}{(\lambda+1)^p},
\end{equation}
again in accordance with the usual harmonic numbers.  Equations~(\ref{Psi-a}),~(\ref{polygamma}) and the fact that $\gamma=-\psi(1)$ further ensure that if $\lambda \in \mathbb{N}$ then $H_{\lambda}^{(p)}$ reduces to the harmonic number defined by~(\ref{HN-ext}).  In particular~(\ref{FlajSal3}) generalizes~(\ref{FlajSal2}) since
\begin{equation*}
\mathbb{S}_{p,t} =\mathbb{S}_{p,t}(1), \qquad \qquad \mathbb{S}_{p,t}^{+-} =\mathbb{S}_{p,t}^{+-}(1).
\end{equation*}
We also mention the \emph{duplication} (or multiplication) formula
\[\textstyle
\psi^{(p)}(z) +\psi^{(p)}(z+\frac{1}{2})  = 2^{p+1}\psi^{(p)}(2z) \qquad p\geq 1.
\]
Appealing to the recurrence relations, this implies
\begin{equation}\label{eq-doubling}
H_{\lambda-\frac{1}{2}}^{(p)} +H_{\lambda}^{(p)} = 2^p\bigl(H_{2\lambda}^{(p)}-\eta(p)\bigr),
\end{equation}
where the \emph{Dirichlet eta function}, $\eta$, is defined by 
\begin{equation*}
\eta(s):=\sum_{n\geq 1}\,\frac{(-1)^{n+1}}{n^{s}} \qquad\text{for\ } \Re (s)>0,\quad \text{
with }\eta (1)=\log 2.
\end{equation*} 
It satisfies $\eta(s)= (1-2^{1-s})\zeta(s)$ for $s\neq 1$ and thus extends analytically; in particular $\eta(0)=\frac{1}{2}$.  We also have \cite[Section 2.3, p. 172, Eq. (49)]{Sr-Ch-12}:
\begin{equation}\label{eq-eta-integral}
\eta(s) = \frac{1}{\Gamma(s)}\int_0^{\infty} \frac{x^{s-1}}{1+e^x}\,\mathrm{d}x,
\end{equation}
which gives the Dirichlet eta function as a Mellin transform.
The \emph{polylogarithm
function} $\mathrm{Li}_{m}$ of \emph{order} $m$ is defined by (see, e.g.,~\cite[p. 198]{Sr-Ch-12}):
\begin{equation}
\aligned\mathrm{Li}_{m}(z)=\sum_{k=1}^{\infty }\frac{z^{k}}{k^{m}}\qquad
|z|\leq 1.
\endaligned  \label{polog1.0}
\end{equation}
Note that $\mathrm{Li}_{1}(z)=-\log(1-z)$. The polylogarithm function satisfies an important functional equation~(\ref{eq-Jonq-positive}) below.  To state it, recall the \emph{Bernoulli numbers} defined by the generating function
\begin{equation*}
\frac{z~}{e^{z}-1}=\sum_{j\geq 0}^{{}}\frac{z^{j}}{j!}B_{j},\qquad \text{ where }
\vert z\vert <2\pi ,
\end{equation*}
so that 
\begin{equation*}
B_{0}=1,\ B_{2n+1}=0,~n\in 
\mathbb{N}
,\text{ with }B_{1}=\textstyle-\frac{1}{2}.
\end{equation*}
More generally the \emph{Bernoulli polynomials} $\mathbf{B}_{j}(t)$ are defined by the generating function
\begin{equation*}
\frac{z\,e^{tz}}{e^{z}-1}=\sum_{j\geq 0}\frac{z^{j}}{j!}\, \mathbf{B}_{j}(
t),\quad \text{ where }\vert z\vert <2\pi,
\end{equation*}
and have the representation (see, e.g.,~\cite[Section 1.7]{SriChoi2001}):
\begin{equation*}
\mathbf{B}_{m}(t) =\sum_{j=0}^{m}\binom{m}{\,j\,} B_{j}\,t^{m-j} =\sum_{j=0}^{m}\binom{m}{\,j\,} B_{m-j}\,t^{j}. 
\end{equation*}
In particular, $B_{m}=\mathbf{B}_{m}(0)$. 

The following important identity is called \emph{Jonqui\`{e}re's relation} 
\cite{Jonqui}. In its general form it provides the analytic continuation of $\mathrm{Li}_{s}(z)$ in~(\ref{polog1.0}) outside its circle of convergence $|z|=1$. See Lewin\footnote{Note that Lewin's $B_n$ is $(-1)^{n+1}B_{2n}$ in our notation.}~\cite[p.~299 (4), (6)]{Lewin}, and also, e.g.,~\cite[pp. 30--31]{Erdel} and \cite[pp. 197--198]{Sr-Ch-12}).\footnote{The spelling ``Joncqui\`{e}re's relation'' also occurs in the literature; see Erd\'{e}lyi et al.~\cite{Erdel}, p. 31.}  In particular for real~$x$ with $0<x<1$ and $m\in \mathbb{N}$ Jonqui\`{e}re's relation gives
\begin{eqnarray}
\mathrm{Li}_{m}(x)+ (-1)^{m} \,\mathrm{Li}_{m} \bigl(\textstyle \frac{1}{x}\bigr) &\!\!\!=\!\!\!& %
-\frac{(2\pi i)^{m}}{m!} \mathbf{B}_{m} \Bigl(\frac{\log x}{2\pi i} \Bigr) \\ %
&\!\!\!=\!\!\!&-\frac{1}{m!} \sum\limits_{j=0}^{m} \binom{m}{\, j \,} B_{j} (2\pi i)^{j} \log^{m -j}(x) \notag \\ 
&\!\!\!=\!\!\!&  \phantom{-}2\sum\limits_{j=0}^{\lfloor \frac{m}{2} \rfloor} \frac{\log^{m-2j}(x)}{(m-2j)!} \zeta(2j) \;\;\; %
+i\,\pi\frac{\log^{m-1}(x)}{(m-1)!},\nonumber\\\label{eq-Jonq-positive} \\ 
\text{\hspace{-3mm} and}\quad \mathrm{Li}_{m}(-x)+(-1)^{m} \,\mathrm{Li}_{m} \bigl(\textstyle -\frac{1}{x} \displaystyle 
\bigr) &\!\!\!=\!\!\!& -2\sum\limits_{j=0}^{\lfloor \frac{m}{2} \rfloor } \frac{\log ^{m-2j}(x)}{(m-2j)!} \,\eta(2j).
\label{eq-Jonq-negative}
\end{eqnarray}

\section{The Main Integral with Parameters}

We now prove Theorem~\ref{THM1}.  The main ingredients are Jonqui\`{e}re's relation~(\ref{eq-Jonq-negative}), and the following results of Sofo and Bat\i r~\cite[Thm.~2.1, Thm.~3.5]{BatirSofo}.  They proved: for real $c\geq -2$ and $q>0$, and for $p$,~$t \in \mathbb{N}\cup\{0\}$
\begin{equation} \label{eq-K}
\mathbf{K}^{\pm}(c,p,q,t) := \!\int_0^1 \!\frac{x^{c} \log^p x\;\; \mathrm{Li}_t(\pm\, x^q)}{1+x} \,\mathrm{d}x = \frac{(-1)^p\, p!}{2^{p+1}} \sum_{n=1}^{\infty} \frac{(\pm 1)^{n}}{n^t} \Bigl[H^{(p+1)}_{\frac{qn +{c}}{2}} -H^{(p+1)}_{\frac{qn +{c}-1}{2}} \Bigr]. 
\end{equation}
 
\begin{proof}[Proof of Theorem~\ref{THM1}]
Starting from the integral $\Ii{-}{a}{b}{p}{t} $ of~(\ref{ONE}) and using the substitution $y=e^{bx}$, we have 
\begin{equation*}
\Ii{-}{a}{b}{p}{t} =\int_{-\infty }^{\infty }\,\frac{x^{p}~\mathrm{Li}_{t}(-e^{ax})}{1+e^{bx}}\,
\mathrm{d}x= \frac{\sign(b)}{b^{p+1}}\int_{0}^{\infty }\,\frac{\log ^{p}y~\mathrm{Li
}_{t}(-y^{q})}{y(1+y) }\,\mathrm{d}y.
\end{equation*}
We split the integral on the right into $\int_0^1 + \int_1^{\infty}$.
In the second integral we utilize the substitution $y=1/z$ (then rename $z=y$).  Multiplying through by $\sign(b)\,b^{p+1}$ and writing $\Inormalized{-}{a}{b}{p}{t} = \sign(b)\,b^{p+1}\, \Ii{-}{a}{b}{p}{t}$ this gives
\[
\Inormalized{-}{a}{b}{p}{t} =
\int_{0}^{1}\, \frac{\log^{p}y~\mathrm{Li}_{t}(-y^{q})}{y(1+y)}\,\mathrm{d}
y \;\;+(-1)^{p}\int_{0}^{1}\,\frac{\log^{p}y~\mathrm{Li}_{t}(-\frac{1}{y^{q}})}{1+y}\,\mathrm{d}y.
\]
On writing $\frac{1}{y(1+y)} = \frac{1}{y} +\frac{\;(-1)^{p+t}}{1+y} -\frac{1+(-1)^{p+t}}{1+y}$ we obtain
\begin{multline*}
\!\!\Inormalized{-}{a}{b}{p}{t}
=\! \int_{0}^{1}\frac{\log ^{p}y~\mathrm{Li}
_{t}(-y^{q})}{y}\,\mathrm{d}y
\;+\;(-1)^{p+t}\!\!\int_{0}^{1}\frac{\log^{p}y\, \bigl[\mathrm{Li}_{t}(-y^{q}) +(-1)^{t}\,\mathrm{Li}_{t}(
\frac{-1}{\;y^{q}})\bigr]}{1+y}\,\mathrm{d}y \\ 
\;\;-\bigl(1 +(-1)^{p+t}\,\bigr)\int_{0}^{1}\frac{\log^{p}y~\mathrm{Li}
_{t}(-y^{q})}{1+y}\,\mathrm{d}y.
\end{multline*}
We now apply Jonqui\`{e}re's relation~(\ref{eq-Jonq-negative}), obtaining
\begin{multline}
\Inormalized{-}{a}{b}{p}{t} =  \int_{0}^{1}\,\frac{\log ^{p}y~\mathrm{Li}_{t}(-y^{q})}{y}\,\mathrm{d}y
\;-\;\bigl(1+(-1) ^{p+t}\,\bigr) \int_{0}^{1}\,\frac{\log ^{p}y~\mathrm{Li}_{t}(-y^{q})}{1+y}\,\mathrm{d}y \\
-2(-1) ^{p+t}\sum\limits_{j=0}^{\lfloor \frac{t}{2} \rfloor }
\frac{q^{t-2j}}{(t-2j) !}\, \eta (2j)
\int_{0}^{1}\,\frac{\log ^{p+t-2j}y~}{1+y}\mathrm{d}y. \label{eq-StorySoFar}
\end{multline}
We deal with each of the integrals in turn.  The integrals in the sum are well known; putting $y=e^{-x}$ in the integral definition~(\ref{eq-eta-integral}) of $\eta$ yields
\begin{equation}\label{eq-logint}
\int_{0}^{1}\,\frac{\log^{m}y~}{1+y}\mathrm{d}y = (-1)^m\, m!\,\eta(m+1).
\end{equation}
The first integral in~(\ref{eq-StorySoFar}) is found by putting ${c}=-1$ and ${c}=0$ in~(\ref{eq-K}) and applying~(\ref{eq-difference-harmonic}):
\begin{equation}\label{eq-first-integral}
\int_{0}^{1} \frac{\log^p x\;\; \mathrm{Li}_t(-x^q)}{x} \,\mathrm{d}x \;=\; \mathbf{K}^{-}(0,p,q,t) +\mathbf{K}^{-}(-1,p,q,t) \;=\; \frac{(-1) ^{p+1}p!}{q^{p+1}}\mathrm{\eta }(p+t+1).
\end{equation}
The remaining (middle) integral in~(\ref{eq-StorySoFar}) is just $\mathbf{K}^{-}(0,p,q,t)$ in~(\ref{eq-K}).  We can simplify it as follows:
\[\renewcommand{\arraystretch}{2.4}\begin{array}{lcl}
\mathbf{K}^{-}(0,p,q,t) &=& \displaystyle \int_0^1 \!\frac{\log^p x\;\; \mathrm{Li}_t(-x^q)}{1+x} \,\mathrm{d}x \\
&\stackrel{(\ref{eq-K})}{=}& \displaystyle \frac{(-1)^{p}\,p!}{2^{p+1}} \sum_{n=1}^{\infty} \frac{(-1)^{n+1}}{n^t} \Bigl[H^{(p+1)}_{\frac{qn-1}{2}} -H^{(p+1)}_{\frac{qn}{2}}\Bigr]  \\   
&\stackrel{(\ref{eq-doubling})}{=} & \displaystyle\frac{(-1)^{p}\,p!}{2^{p+1}}   \sum_{n=1}^{\infty} \frac{(-1)^{n+1}}{n^t} \Bigl[2^{p+1}\bigl(H_{qn}^{(p+1)}-\eta(p+1)\bigr) -2H^{(p+1)}_{\frac{qn}{2}}\Bigr]. \\  
\end{array}
\]
Utilizing the notation~(\ref{FlajSal3}) we can write
\begin{align}
-\bigl(1+(-1)^{p+t}\bigr) \mathbf{K}^{-}(0,p,q,t) =&-\bigl((-1)^p+(-1)^t\bigr)\,p!\nonumber\\
&\;\;\;\;\;\;\times\Bigl[\mathbb{S}_{p+1,t}^{+-}(q) -2^{-p}\,\mathbb{S}_{p+1,t}^{+-}(\textstyle \frac{q}{2}) 
-\eta(p+1)\,\eta(t) \Bigr].\quad \label{eq-Kint}
\end{align}
Substituting~(\ref{eq-logint}),~(\ref{eq-first-integral}) and~(\ref{eq-Kint}) into~(\ref{eq-StorySoFar}) we obtain:
\begin{multline*}
\frac{\sign(b)}{p!}\,b^{p+1}\!\! \int_{-\infty }^{\infty }\!\frac{x^{p}~\mathrm{Li}_{t}(-e^{ax})}{1+e^{bx}}\,\mathrm{d}x = 
-2\sum\limits_{j=0}^{\lfloor \frac{t}{2} \rfloor }q^{t-2j}\binom{p\!+\!t\!-\!2j}{p} \eta (p + t + 1 -2j)\,\eta(2j) \\ +\frac{(-1)^{p+1}}{q^{p+1}}\eta
(p+t+1) 
-\bigl((-1)^p+(-1)^t\bigr)\Bigl[\mathbb{S}_{p+1,t}^{+-}(q) -2^{-p}\,\mathbb{S}_{p+1,t}^{+-}(\textstyle \frac{q}{2}) -\eta(p+1)\,\eta(t) \Bigr].
\end{multline*}
This completes the proof.
\end{proof}

The special cases $p=0$ and $p=t$ are highlighted in the next corollary.

\begin{corollary}
If $p=0$ and $t$ is an odd integer:
\begin{equation*}
|b|\; \Ii{-}{a}{b}{0}{t}  =|b| \!\!\int_{-\infty }^{\infty }\frac{~\mathrm{Li}
_{t}(-e^{ax})}{1+e^{bx}}\,\mathrm{d}x  
= -\frac{1}{q}\eta(t+1) -2\sum\limits_{j=0}^{\frac{t-1}{2}
} q^{t-2j}\, \eta(2j)\, \eta(t+1-2j).
\end{equation*}
If $p=0$ and $t$ is an even integer:
\begin{multline*}
|b|\; \Ii{-}{a}{b}{0}{t}  = |b|\!\! \int_{-\infty }^{\infty }\frac{~\mathrm{Li}_{t}(-e^{ax})}{1+e^{bx}}\, \mathrm{d}x  
=-\frac{1}{q}\eta (t+1)   -2\sum\limits_{j=0}^{\frac{t}{2}} q^{t-2j}\,\eta(2j)\, \eta(t+1-2j)  \\ 
+2\,\eta(t) \log 2 \;-2\,\mathbb{S}_{1,t}^{+-}(q) \;+2\,\mathbb{S}_{1,t}^{+-} \bigl( \textstyle\frac{q}{2}).  
\end{multline*}

For $p=t$ and $q=1$,
\begin{multline*}
\frac{(-a)^{t+1}}{t!} \Ii{-}{a}{a}{t}{t} = \eta(2t+1)-2\eta(t+1)\,\eta(t)  \;\;+2(-1)^t\sum_{j=0}^{\lfloor \frac{t}{2}\rfloor} \binom{2t\!-\!2j}{t}\eta(2t+1-2j)\,\eta(2j) \\+2\,\mathbb{S}^{+-}_{t+1,t} \; -2^{1-t}\,\mathbb{S}^{+-}_{t+1,t}(\textstyle \frac{1}{2}).
\end{multline*}
\end{corollary}

\begin{proof}
This follows directly from Theorem~\ref{THM1}.
\end{proof}

\begin{remark}
With $a=b=1$
\begin{equation*}
\Ii{-}{1}{1}{0}{t} =\int_{-\infty }^{\infty }\,\frac{~\mathrm{Li}
_{t}(-e^{x})}{1+e^{x}}\,\mathrm{d}x=\int_{0}^{\infty }\,\frac{~\mathrm{Li}
_{t}(-y)}{y(1+y) }\,\mathrm{d}y.
\end{equation*}
The polylogarithm function can be expressed as a Fermi-Dirac integral, via the representation~\cite{SriChoi2001}
\begin{equation*}
\mathrm{Li}_{t}(-y)=-\frac{1}{\Gamma (t) }\int_{0}^{\infty }\,
\frac{~x^{t-1}}{y^{-1}e^{x}+1}\,\mathrm{d}x  
\end{equation*}
so that
\begin{align*}
\Ii{-}{1}{1}{0}{t}  =& -\frac{1}{\Gamma (t) }
\int_{0}^{\infty} x^{t-1} \int_{0}^{\infty }\,\frac{~\mathrm{d}y}{
(1+y) (e^{x}+y) }\, \,\mathrm{d}x \\ %
=& -\frac{1}{\Gamma (t) }\int_{0}^{\infty }\,\frac{~x^{t}}{
e^{x}-1}\,\,\mathrm{d}x=-t\,\zeta(t+1).
\end{align*}
If $t$ is even and we substitute the above into~(\ref{INTthm1.4}) we obtain
\begin{equation}
2\,\mathbb{S}_{1,t}^{+-}\bigl( \textstyle \frac{1}{2} \displaystyle \bigr) \;=\; 2\,\mathbb{S}_{1,t}^{+-} +2\eta(t+1) -t\,\zeta(t+1)  -2\eta(t) \log 2 +2\sum\limits_{j=1}^{\frac{t}{2}}\eta(2j)\, \eta (t+1-2j).  \label{q1}
\end{equation}
R. Sitaramachandrarao~\cite[3.17]{SitaR} has given the identity
\begin{equation}
2\,\mathbb{S}_{1,t}^{+-} = (t+1) \eta(t+1) -\zeta(t+1) -2\sum\limits_{j=1}^{\frac{t}{2}-1}\eta(2j)\,
\zeta(t+1-2j).  \label{q2}
\end{equation}
Substituting~(\ref{q2}) into~(\ref{q1}) provides the new linear harmonic
Euler sum identity
\begin{equation}
2\,\mathbb{S}_{1,t}^{+-}\bigl(\textstyle \frac{1}{2}\bigr) =(t+3) \eta(t+1) -(t+1) \zeta(t+1)  
-2^{1-t}\sum\limits_{j=1}^{\frac{t}{2}-1}2^{2j}\eta (2j)\,
\zeta(t+1-2j),\qquad\text{$t$ even}.  \label{q3}
\end{equation}
\end{remark}

\begin{example}
Let $p=2p^{\ast}$ (and rename $p^{\ast }=p$), $t=1$, $q=\frac{a}{b}$ with $ab>0$. Then
\begin{align*}
\sign(b)\,b^{2p+1} \Ii{-}{a}{b}{2p}{1} =&\sign(b)\,b^{2p+1}\!\!\! \int_{-\infty }^{\infty} \frac{x^{2p} ~\mathrm{Li}_{1}(-e^{ax})}{1+e^{bx}}\,\mathrm{d}x \\ 
=&-\biggl[\frac{(2p)!}{q^{2p+1}} +q(2p+1)!\biggr]\eta (2p+2).
\end{align*}
Rearranging, we obtain, for $ab>0$
\begin{equation*}
\eta (2p+2) \;=\; -\frac{\sign(b)\,a^{2p+1}}{(2p)! \bigl(1 +q^{2p+2}(2p+1)\bigr)} \int_{-\infty }^{\infty }\,\frac{x^{2p}~\mathrm{Li}_{{1}}(-e^{ax})}{1+e^{bx}}\,\mathrm{d}x
\end{equation*}
which, in the special case of $a=b=q=1$, we recover the result  
\cite[(1.8) p.~2]{KUMAR}. 

Let $p=2p^{\ast }-1$ (and rename $p^{\ast }=p$), $t=1$, $q=\frac{a}{b}$ with $ab>0$.  Then 
\begin{eqnarray*}
\sign(b)\, b^{2p+1} \Ii{-}{a}{b}{2p-1}{1} &=& \sign(b)\,b^{2p+1} \!\!\int_{-\infty }^{\infty }\,\frac{x^{2p-1}~\mathrm{Li}_{1}(-e^{ax})}{1+e^{bx}}\, \mathrm{d}x \\
&=& \frac{(2p-1) !}{q^{2p}}\, \eta(2p+1) -2(2p-1)!\,\eta(2p)\, \log 2 \\
&&-q(2p)!\, \eta(2p+1) +2(2p-1)! \\
&&\;\;\;\;\;\;\;\;\;\;\;\;\times\bigl(
\mathbb{S}_{2p,1}^{+-}(q) -2^{1-2p}\, \mathbb{S}_{2p,1}^{+-}(\textstyle \frac{q}{2}) \bigr).
\end{eqnarray*}
The case of $a=b=q=t=1$ and $p\in \mathbb{N}$ gives
\begin{equation*}
\Ii{-}{1}{1}{p}{1} \;=\;\int_{-\infty}^{\infty}\, \frac{x^{p} ~\mathrm{Li}_{1}(-e^{ax})}{1+e^{bx}} \,\mathrm{d}x \;=\; \int_{0}^{\infty}\frac{~\log^{p}(y) \mathrm{Li}_{1}(-y)}{y(1+y) }\,\mathrm{d}y.
\end{equation*}
Again utilizing the Fermi-Dirac integral
, we put 
\begin{eqnarray*}
\mathbf{X}(a)  &:=& \int_{0}^{\infty }\frac{~y^{a}\,\mathrm{Li}_{1}(-y)}{y(1+y) }\,\mathrm{d}y,\qquad \vert a\vert <1 \\
&=&-\frac{1}{\Gamma (1) }\int_{0}^{\infty
}x^{1-1}\int_{0}^{\infty }\frac{~y^{a}}{(1+y) (
e^{x}+y) }\,\mathrm{d}y\, \mathrm{d}x \\
&\stackrel{(*)}{=}& \frac{\pi }{\Gamma (1) }\int_{0}^{\infty }\frac{~(
1-e^{ax}) \csc (a\pi) }{(e^{x}-1) }\,\mathrm{d}x \\
&\stackrel{(\dag)}{=}&\pi \csc (a\pi) \bigl(\gamma +\psi (1-a)\bigr).
\end{eqnarray*}
See~\cite[3.223(1), 3.311(5)]{Table-Integrals} for~($*$), ($\dag$) respectively. Therefore
\begin{eqnarray*}
\Ii{-}{1}{1}{p}{1}   &=& \int_{0}^{\infty }\frac{\,\log^{p}(y) \mathrm{Li}_{1}(-y)}{y(1+y)}\,\mathrm{d}y \;=\; \underset{a\rightarrow 0}{\lim }\Bigl(\frac{\partial
^{p}}{\partial a^{p}}\mathbf{X}(a) \Bigr) \\
&=&\underset{a\rightarrow 0}{\lim }\Bigl(\frac{\partial ^{p}}{\partial a^{p}}\bigl(
\pi \csc (a\pi) (\gamma +\psi (1-a))
\bigr) \Bigr)  \\
&=& (-1)\, p!\,\Bigr(\zeta (p+2) +2\sum\limits_{j=1}^{\lfloor \frac{p}{2} \rfloor }\eta (2j) \zeta (p+2-2j) \Bigr) 
\end{eqnarray*}
and this confirms the result obtained by Li and Chu~\cite[p.~572]{LIChu}.

For $a=6$, $b=1$, $p=1$, $t=1$ we obtain the elegant result for $\Ii{-}{6}{1}{1}{1}$\;:
\begin{multline*}
\int_{-\infty }^{\infty }\,\frac{x~\mathrm{Li}_{1}(-e^{6x})}{1+e^{x}}\,
\mathrm{d}x \textstyle = \frac{1}{48}\zeta(3) -\frac{7}{3}\pi G -\zeta(2)\bigl( 4\log \bigl(1 +\sqrt{3}\,) -\frac{5}{4}\log 2\bigr) \\ 
+\frac{\pi}{72\sqrt{3}} \textstyle\Bigl( \psi^{(1)}(\frac{1}{6}) +\psi^{(1)}(\frac{1}{3}) -\psi^{(1)}(\frac{2}{3}) -\psi^{(1)}(\frac{5}{6}) \Bigr),
\end{multline*}
where $G$ is Catalan's constant.
\end{example}

In the next theorem we establish the complementary version of Theorem~\ref{THM1}.

\begin{theorem}
\label{THM2} Let $t \in \mathbb{N}$, $p\in \mathbb{N}\cup\{0\}$ and $a$, $b \in \mathbb{R}$ with $ab>0$.   Define $q = \frac{a}{b}$.  Then, the following hold:

\noindent  If $p+t$ is odd 
\begin{align*}
\frac{\sign(b)}{p!}\, b^{p+1}\,\Ii{+}{a}{b}{p}{t} 
=& \frac{(-1)^{p}}{q^{p+1}}\zeta(p+t+1) \\
&+2\sum_{j=0}^{\lfloor \frac{t}{2}\rfloor} q^{t-2j} \binom{p\!+\!t\!-\!2j}{p}  \eta(p+t+1-2j)\, \zeta(2j) \\
&-q^{t-1}\binom{p\!+\!t\!-\!1}{p} \eta(p+t)\,\pi \, i. 
\end{align*}
If $p+t$ is even
\begin{align}
\frac{\sign(b)}{p!}\, b^{p+1}\,\Ii{+}{a}{b}{p}{t} %
=& \frac{(-1)^{p}}{q^{p+1}}\zeta(p+t+1) \nonumber\\
&+2\sum_{j=0}^{\lfloor \frac{t}{2}\rfloor} q^{t-2j} \binom{p\!+\!t\!-\!2j}{p} \eta(p+t+1-2j)\, \zeta(2j)  \nonumber\\ 
&+2(-1)^{p}\sum_{n=1}^{\infty} \sum_{k=1}^{\infty} \frac{(-1)^{k}}{n^t(qn+k)^{p+1}} \nonumber\\
&-q^{t-1}\binom{p+t-1}{p} \eta(p+t)\,\pi \, i.\label{eq-pplusteven-general}
\end{align}
Furthermore, if $p+t$ is even and $t>1$ then~(\ref{eq-pplusteven-general}) can be written as
\begin{align}
\frac{\sign(b)}{p!}\, b^{p+1}\,\Ii{+}{a}{b}{p}{t} %
=& \frac{(-1)^{p}}{q^{p+1}}\zeta(p+t+1) \nonumber\\
&+2\sum_{j=0}^{\lfloor \frac{t}{2}\rfloor} q^{t-2j} \binom{p\!+\!t\!-\!2j}{p} \eta(p+t+1-2j)\, \zeta(2j)  \nonumber\\ 
&+2(-1)^{p} \Bigl[ \mathbb{S}_{p+1,t}(q) -2^{-p} \,\mathbb{S}_{p+1,t}\bigl( \textstyle \frac{q}{2} \displaystyle\bigr) -\eta(p+1)\,\zeta(t)\Bigr] \nonumber\\
&-q^{t-1}\binom{p+t-1}{p} \eta(p+t)\,\pi \, i. 
\end{align}
\end{theorem}
In all cases, all the terms are real except the last.

\begin{proof} The proof is identical to the proof of Theorem~\ref{THM1}, except for two adjustments for sign.  Namely, we use the positive version of Jonqui\`{e}re's relation~(\ref{eq-Jonq-positive}) instead of the negative version~(\ref{eq-Jonq-negative}), and $\mathbf{K}^{+}$ in~(\ref{eq-K}) instead of $\mathbf{K}^{-}$.

There is one complication: the simplification of the harmonic sums arising from~(\ref{eq-K}) cannot be performed if $t=1$ since the convergent series (whose summands are differences) does not split into (the difference of) several convergent series.  See Example~\ref{ex-problem-case} below.  Instead, to state~(\ref{eq-pplusteven-general}) we have just appealed to the definition of generalized harmonic numbers~(\ref{G-HZ}).
\end{proof}

\begin{example}\label{ex-problem-case}
In the $\mathbf{I}^+$ case, if $q=2$ and $p=t=1$ the sum arising from~(\ref{eq-K}) is

\[
S = -\frac{1}{2}\sum\limits_n \frac{1}{n}\Bigl[H_{n-\frac{1}{2}}^{(2)} -H_{n}^{(2)} \Bigr]%
%
%
\;\;\stackrel{(\ref{eq-doubling})}{=}\;\; \sum\limits_n \frac{1}{n}\Bigl[H_{n}^{(2)} -2H_{2n}^{(2)} +\zeta(2) \Bigr]. %
\]
Of course we cannot split the right hand side into three convergent sums.  However by Abel summation we find
\[
S= 3\zeta(3)-\frac{1}{2}\pi^2\log(2).
\]

Altogether
\[
12\int_{-\infty}^{\infty} \frac{x\log(1-e^{2x})}{1+e^x}\mathrm{d}x  = 3\zeta(3) +6\pi^2\log(2) +\pi^3\,i.
\]
\end{example}

\medskip

Theorems~\ref{THM1} and~\ref{THM2} have a similar appearance.  At the cost of a little complexity, they can be combined into the single statement, Theorem~\ref{THM3} below.

\medskip

\noindent \emph{In the next result, either every $\pm$ sign is to be interpreted as $+$, or every $\pm$ sign is interpreted as $-$.}
\begin{theorem}\label{THM3}
Let $a$, $b\in \mathbb{R}$ with $ab>0$, and let $p \in \mathbb{N}\cup\{0\}$ and $p \in \mathbb{N}$.  Define $q=a/b$ and let $\xi^{+}(z)=\zeta(z)$ and $\zeta^{-}(z)=\eta(z)$.
Then 
\[
\Ii{\pm}{a}{b}{p}{t} = p!\,\frac{\sign(b)}{b^{p+1}}\,\Bigl[A^{\pm}+ B^{\pm}  \;+\; C^{\pm} i\Bigr],
\]
where $A^{\pm}$, $B^{\pm}$, $C^{\pm}$ are the real numbers given below:
\begin{eqnarray}
A^{\pm} &\!=\!& \pm \frac{(-1)^{p}}{q^{p+1}}\xi^{\pm}(p+t+1) \;\;\pm 2\sum_{j=0}^{\lfloor \frac{t}{2}\rfloor} q^{t-2j} \binom{p\!+\!t\!-\!2j}{p}  \eta(p+t+1-2j)\, \xi^{\pm}(2j), \notag \\ %
B^{\pm} &\!=\!& \bigl[(-1)^{p} \!+\!(-1)^t\bigr] \sum_{n=1}^{\infty}(\pm 1)^{n} \sum_{k=1}^{\infty} \frac{(-1)^{k}}{n^t(qn+k)^{p+1}}, \label{eq-safe-case} \\[1mm] %
&\!=\!& (\pm)\bigl[(-1)^{p} \!+\!(-1)^t\bigr]\Bigl[ \mathbb{S}_{p+1,t}^{+\pm}(q) -2^{-p} \,\mathbb{S}_{p+1,t}^{+\pm}\bigl( \textstyle \frac{q}{2} \bigr) -\eta(p+1)\,\xi^{\pm}(t)\Bigr] \quad \text{if\ }(t,\pm) \neq (1,+),\nonumber\\ \label{eq-t-greater-1} \\ %
C^{+} &\!=\!&-q^{t-1}\binom{p\!+\!t\!-\!1}{p}\, \eta(p+t)\,\pi, \notag \\ %
C^{-} &\!=\!& 0. \notag
\end{eqnarray}
\end{theorem}

Notes:
\begin{itemize}
\item In the $I^-$ case the integral is real.  In the $I^+$ case it has a single imaginary term $C^+i$.

\item The initial term in $A^{\pm}$ can be viewed as a term in the succeeding sum with index ``$j=\frac{1}{2}(p+t+1)$''.

\item If $p+t$ is odd then $B^{\pm}=0$.

\item The sums in~(\ref{eq-t-greater-1}) do not exist if $t=1$  However~(\ref{eq-safe-case}) still gives $B^{\pm}$ in this case.
\end{itemize}

The special cases of $p=0$ and $p=t$ are highlighted in the next corollary.

\begin{corollary}
\noindent  If $p=0$ and $t$ is odd 
\begin{align*}
|b|\,\Ii{+}{a}{b}{0}{t} 
=&|b| \!\!\int_{-\infty }^{\infty } \,\frac{\mathrm{Li}_{t}(e^{ax})}{1+e^{bx}} \,\mathrm{d}x = \frac{1}{q}\zeta(t+1)  \;\;+ 2\sum_{j=0}^{\lfloor \frac{t}{2}\rfloor} q^{t-2j} \eta(t+1-2j)\, \zeta(2j)\\
&-q^{t-1}\eta(t)\,\pi \, i.
\end{align*}

If $p=0$ and $t$ is even
\begin{multline*}
|b|\,\Ii{+}{a}{b}{0}{t} 
=|b| \!\!\int_{-\infty }^{\infty } \,\frac{\mathrm{Li}_{t}(e^{ax})}{1+e^{bx}} \,\mathrm{d}x = \frac{1}{q}\zeta(t+1)  \;\;+ 2\sum_{j=0}^{\lfloor \frac{t}{2}\rfloor} q^{t-2j} \eta(t+1-2j)\, \zeta(2j)\;\; \\ %
+2\Bigl[ \mathbb{S}_{1,t}(q) -\,\mathbb{S}_{1,t}\bigl( \textstyle \frac{q}{2} \displaystyle\bigr) -\zeta(t)\log 2\Bigr] %
-q^{t-1}\eta(t)\,\pi \, i.
\end{multline*}

If $p=t \in \mathbb{N}$ is even, then
\begin{multline*}
\frac{\sign(b)}{t!}\, b^{t+1}\,\Ii{+}{a}{b}{t}{t} %
=\frac{\sign(b)\, b^{t+1}}{t!} \!\!\int_{-\infty }^{\infty } \,\frac{x^{t} ~\mathrm{Li}_{t}(e^{ax})}{1+e^{bx}} \,\mathrm{d}x \\
= \frac{(-1)^{t}}{q^{t+1}}\zeta(2t+1) \;\;+ 2\sum_{j=0}^{\lfloor \frac{t}{2}\rfloor} q^{t-2j} \binom{2t\!-\!2j}{t} \eta(2t+1-2j)\, \zeta(2j)  \\ %
+2\Bigl[ \mathbb{S}_{t+1,t}(q) -2^{-t} \,\mathbb{S}_{t+1,t}\bigl( \textstyle \frac{q}{2} \displaystyle\bigr) -\eta(t+1)\,\zeta(t)\Bigr] 
\quad -q^{t-1}\binom{2t-1}{t} \eta(2t)\,\pi \, i.
\end{multline*}

\end{corollary}

\begin{proof}
The result follows directly from Theorem~\ref{THM2}.
\end{proof}

\begin{example}
Let $p=2p^{\ast}$ (and rename $p^{\ast }=p$), $t=1$, $q=\frac{a}{b}$, where $a$, $b >0$.  Then the following identity holds true:
\[
\frac{|b|^{2p+1}}{(2p)!} \Ii{+}{a}{b}{2p}{1} 
= \frac{1}{q^{2p+1}}\zeta(2p+2)  \;-q(2p+1)(1-2^{-1-2p})\zeta(2p+2)
 \;-\eta(2p+1)\,\pi \, i.
\]
Isolating the imaginary and real parts of the integral we obtain the
intriguing identities:
\[
\eta(2p+1) =\frac{|b|^{2p+1}}{\pi (2p)!}\, \Im\! \left(\,\int_{-\infty}^{\infty} \frac{x^{2p}\,\log(1-e^{ax})}{1+e^{bx}}\,\mathrm{d}x\!\right).
\]
and 
\[
\zeta(2p+2) = -\frac{|a|^{2p+1}}{(2p)![1- q^{2p+2}(2p+1)(1-2^{-1-2p})]}   \Re\! \left(\,\int_{-\infty}^{\infty} \frac{x^{2p}\,\log(1-e^{ax})}{1+e^{bx}}\,\mathrm{d}x\!\right).
\]
\end{example}

\begin{example}
Let $a=2$, $b=1$, $t=4$ and $p=2$.  On application of Theorem~\ref{THM2} the integral $\Ii{+}{2}{1}{2}{4}$ evaluates to a sum of $\zeta$ (or $\eta$) terms, together with $\mathbb{S}_{3,4}$ and  $\mathbb{S}_{3,4}(2)$. The evaluation 
\[
\mathbb{S}_{3,4}=18\zeta(7)-10\zeta(2)\zeta(5)
\]
is due to Borwein~\cite{Borwein}, \cite[Thm~3.1]{FlajSalv}. From Eq. (\ref{FlajSal3}), one has
\begin{equation*}
\mathbb{S}_{3,4}(2)=\sum_{n=1}^{\infty}\frac{H_{2n}^{(3)}}{n^4},
\end{equation*}
and we obtain by decomposition
\begin{equation*}
\mathbb{S}_{3,4}(2)=8\mathbb{S}_{3,4}-8\mathbb{S}_{3,4}^{+-}.
\end{equation*}
From \cite[Thm~7.2, p. 33]{FlajSalv}, we may evaluate
\begin{equation*}
\mathbb{S}_{3,4}^{+-}(2) = \frac{363}{128}\zeta(7) - \frac{9}{8}\zeta(5)\zeta(2).    
\end{equation*}
thereby producing
\begin{equation*}
\mathbb{S}_{3,4}(2) = \frac{1941}{16}\zeta(7) -71\zeta(5)\zeta(2).
\end{equation*}
Upon simplifying we obtain:
\begin{equation*}
\int_{-\infty }^{\infty}\, \frac{x^{2} ~\mathrm{Li}_{4}(e^{2x})}{1+e^{x}} \,\mathrm{d}x = -5\zeta(7) -184\zeta(2)\zeta(5) \;-\frac{31}{189}\pi^7\,i.
\end{equation*}

\end{example}

\section{Concluding Remarks}

We have successfully evaluated the Integral~(\ref{ONE}), shown previously to
have occurred in the course of a work on statistical plasma physics, in the
so-called Sommerfeld temperature-expansion of the electronic entropy. The
representation of both the Riemann zeta and Dirichlet eta functions have
been given in terms of two independent parameters. The new linear harmonic
Euler sum identity~(\ref{q3}) has been identified. We have also generalized
the results of Li and Chu~\cite{LIChu}.

\end{document}